\documentclass[12pt]{article}

 \usepackage{graphics}
 \usepackage{graphicx}
 \usepackage{epsfig}
 \usepackage[all]{xy}

 \usepackage{amssymb}
 \usepackage{amsthm,amsmath}

 \usepackage{latexsym, epsfig, psfrag,eepic,colordvi,bm}
 \usepackage{graphics,color}
 \usepackage{amsmath,amsfonts,amssymb,amscd}
 \usepackage[all]{xy}
 \usepackage{epstopdf}

 \usepackage{amsthm,amsmath,amssymb}

 \usepackage{graphicx}
\usepackage{amsthm,amsmath,amssymb}

\usepackage{graphicx,epsfig}

\usepackage{arydshln}

\usepackage[numbers,sort]{natbib}

\theoremstyle{plain}
\newtheorem{theorem}{Theorem}[section]
\newtheorem{lemma}[theorem]{Lemma}
\newtheorem{corollary}[theorem]{Corollary}
\newtheorem{proposition}[theorem]{Proposition}

\theoremstyle{definition}
\newtheorem{definition}{Definition}[section]
\newtheorem{example}[definition]{Example}

\theoremstyle{remark}
\newtheorem{remark}[theorem]{Remark}

\newcommand{\bai}{\hspace{6pt}}

\usepackage{xspace}
\def\GOE{\ensuremath{\textrm{GoE}}\xspace}
\def\SDS{\ensuremath{\textrm{SDS}}\xspace}

\date{}

\title{Garden-of-Eden states and fixed points\\ of monotone dynamical systems}

\author{Ricky X. F. Chen$^{1}$, Henning S. Mortveit$^{1,2}$, Christian M. Reidys$^{1,2}$\\
	\small $^1$Biocomplexity Institute, Virginia Tech, 1015 Life Science Circle, Blacksburg, VA 24061, USA\\[-0.8ex]
	\small $^2$Dept. of Mathematics, Virginia Tech, Blacksburg, VA 24061, USA\\
		\small\tt chenshu731@sina.com, hmortvei@vt.edu, duck@santafe.edu
}

\begin{document}
\maketitle
\noindent{\bf Abstract.}  In this paper we analyze Garden-of-Eden (\GOE) states and fixed points of
monotone, sequential dynamical systems (\SDS). For any monotone \SDS and fixed update schedule, we
identify a particular set of states, each state being either a \GOE state or reaching a fixed
point, while both determining if a state is a \GOE state and finding out all fixed points are generally hard. As a result, we show that the maximum size of their limit cycles is strictly less than
${n\choose \lfloor n/2 \rfloor}$. We connect these results to the Knaster-Tarski theorem and the LYM
inequality. Finally, we establish that there exist monotone, parallel dynamical systems (PDS) that
cannot be expressed as monotone \SDS, despite the fact that the converse is always true.

\vskip 12pt



\section{Introduction}

The study of dynamical systems is central in modern mathematics and its applications in other fields
such as physics, computer science, and biology. In this paper, we investigate discrete-time
dynamical systems over graphs. Given a graph~$G$, each vertex of~$G$ has a state contained in some
finite set~$P$.  At each time step, some or all of the vertices update their states according to
their respective local functions or updating rules to generate discrete-time dynamics. The precise
manner in which the local functions assemble to a dynamical system map is called the update
mechanism. Examples of classes of such dynamical systems include Boolean networks~\cite{kauf},
Cellular Automata~(CA)~\cite{vonn,wolf}, Hopfield networks~\cite{hopfield}, and sequential dynamical
systems (\SDS\footnote{We will write \SDS in singular as well as plural
  form.})~\cite{rei1,rei2,rei3,rei4,rei5}.

For a graph~$G$ on~$n$ vertices there are $|P|^n < \infty$ distinct states.  As a result, the
iterates of any system state will eventually cycle through a subset of system states, called
\emph{periodic points}. As a result, the system dynamics generates a directed graph, called the
\emph{phase space}, consisting of a collection of disconnected cycles with trees attached at cycle
vertices. Accordingly, the dynamics can be characterized by these cycles (\emph{limit cycles}), the
leaves of the attached trees (Garden-of-Eden (\GOE) states), and the non-leaf tree states (\emph{transient
states}) connecting \GOE states and periodic points.  A limit cycle consisting of only one periodic
point is called a \emph{fixed point}.

Systems composed of monotone functions were studied
in~\cite{goles-mon,laub,and-or,sontag1,sontag2,mon-neural1,mon-neural2}, and systems of linear
functions and monomial functions were analyzed in~\cite{chr,linear} and~\cite{monomial},
respectively.  A number of studies are concerned with the existence and number of \GOE states and
fixed points~\cite{goe1,moore,myhill,aracena,mon-neural2}, as well as the size of limit
cycles~\cite{goles-mon,macmor}.

In this paper, we study dynamical systems having monotone, local functions. One key result
characterizes a set of states that are either a \GOE-state or that eventually reach a fixed point.
Notably, this set depends exclusively on monotonicity and not the particular choice of monotone
local functions.

The paper is organized as follows. In Section~\ref{sec2}, we recall the notation of \SDS and
monotone \SDS. In Section~\ref{sec3}, we study monotone \SDS in detail. Here we establish several
key results and applications. For example, we show that in any monotone \SDS on~$n$ vertices, the
probability of a random state being either a \GOE state or reaching a fixed point under a randomly
chosen update schedule is at least~$\frac{n}{2^{n-1}}$.  Furthermore, we prove that the maximum size
of limit cycles of monotone \SDS is strictly less than~${n\choose \lfloor n/2 \rfloor}$.  We also
refine the LYM inequality and present a finite version of the Knaster-Tarski theorem, from which
it follows that if there exists a non-trivial periodic state then there exist at least two trivial
periodic states (i.e.,~fixed points.)

In Section~\ref{sec4}, we discuss the sequentialization of monotone, parallel dynamical systems
(PDS), that is, the construction of a monotone \SDS whose graph~$G$ has the same number of vertices
and which has exactly the same dynamics as a given monotone PDS.  We prove that there exists
monotone PDS that cannot be sequentialized, and we provide sufficient and necessary conditions for
a monotone PDS to have a monotone sequentialization.

\section{Basic Definitions}\label{sec2}

Let~$G$ be a simple graph with vertex set~$V(G)=\{1,2,\ldots,n\}$ with each vertex having a
state taken from a finite set~$P$.
A state of the system is a tuple~$X=[x_1, x_2,\ldots, x_n]\in P^n$ where the $i$-th coordinate
represents the state of the vertex~$i$.
A function~$f_{i}$ is used to update the state of~$i$ based on the states of its neighbors and
itself.
A permutation~$\pi=\pi_1\pi_2\cdots\pi_n$ of the vertices in~$V(G)$ is called an update schedule.
The system dynamics is generated by sequentially updating the states of all vertices
according to~$\pi$, that is, the vertex labeled~$\pi_{i+1}$ is updated after the vertex labeled
by~$\pi_i$. We denote the associated~\SDS by the triple~$(G,f = (f_i)_i,\pi)$. The map resulting from a single application of this update procedure is denoted~$F_\pi$ (or $(G,f,\pi)$ by abuse of notation).

By inflation, each local function~$f_{i}$ induces the function
\begin{equation*}
	F_{i} \colon P^n \rightarrow P^n, \quad [x_1,\dots, x_i,\dots, x_n]
	\mapsto
	[x_1,\dots, f_{i}, \dots, x_n],
\end{equation*}
where the arguments taken by~$f_{i}$ are the current states of~$i$ and its neighbors, and we have
\begin{equation*}
	F_{\pi}=F_{\pi_n}\circ F_{\pi_{n-1}}\circ \dots \circ F_{\pi_1} \;.
\end{equation*}
As mentioned above, the directed graph on~$P^n$ with directed edges~$\bigl(X, F_{\pi}(X)\bigr)$
where~$X\in P^n$ is called the phase space of the \SDS. A state~$X$ is called a Garden-of-Eden (\GOE) state
if there is no state~$Y$ such that~$F_{\pi}(Y)=X$.
Directed cycles in the phase space are called limit cycles, and states contained in limit cycles are
called attractors or periodic states. States contained in a limited cycle of length no less than
two are called non-trivial periodic states. A state~$X$ is called a fixed point if~$F_{\pi}(X)=X$.

Unless stated otherwise, we will in the following assume~$P=\mathbb{F}_2=\{0,1\}$ with
ordering~$0<1$. Let~`$\leq$' be the partial order on~$\mathbb{F}_2^q$ ($q\geq 1$) given by
\begin{equation*}
	[x_1,x_2,\dots , x_q]\leq [y_1,y_2,\dots, y_q] \;,
\end{equation*}
iff~$x_i\leq y_i$ in~$\mathbb{F}_2$ for all~$1\leq i \leq q$.  We shall denote the
minimum~$[0,0,\ldots, 0]$ by~${\bf 0}$ and the maximum~$[1,1,\ldots, 1]$ by~${\bf 1}$.

\begin{definition} A function
	$g \colon \mathbb{F}_2^q\rightarrow \mathbb{F}_2$ is called monotone
	if for any~$X,Y \in \mathbb{F}_2^q$ with~$X\leq Y$ one has~$g(X)\leq g(Y)$.
	The function~$g$
	is called a simple threshold function if for some fixed~$k\geq 0$
	\begin{equation*}
		\forall X=[x_1,x_2,\ldots, x_q]\in \mathbb{F}_2^q, \quad
		g(X)=	\begin{cases}
			1, & \text{if $\sum_j x_j \geq k$\text{, and}}  \\
			0, & \text{otherwise}.
		\end{cases}
	\end{equation*}
\end{definition}
It is easy to check that the binary functions `AND' and `OR' are simple threshold functions, and
that simple threshold functions are monotone. We will call a (sequential) dynamical system with
monotone, local functions a monotone (sequential) dynamical system.

\section{\GOE States and Fixed Points}\label{sec3}
In this section we analyze \GOE states and fixed points of monotone \SDS. We shall identify a set of
states that are either \GOE states or reach a fixed point, regardless of the particular choice of the
(local) monotone functions.

\begin{lemma}\label{4lem1}
 Let~$(G,f,\pi)$ be a monotone \SDS and~$X,Y$ be two states. Then we have: \\
\phantom\quad (a) $X \leq Y$ implies~$F_{\pi}(X)\leq F_{\pi}(Y)$. \\
\phantom\quad (b) if~$X\leq F_{\pi}(X)$ or~$X\geq F_{\pi}(X)$, then~$X$ reaches a fixed point.\\
\phantom\quad (c) the states contained in a limit cycle form an anti-chain.
\end{lemma}

\proof Using the fact that the functions~$F_{\pi_j}$ are monotone, the relation~$F_{\pi_i}\circ
\dots \circ F_{\pi_2}\circ F_{\pi_1}(X)\leq F_{\pi_i}\circ \dots \circ F_{\pi_2}\circ F_{\pi_1}(Y)$
for $1\leq i \leq n$ follows by induction on~$i$, establishing~{(a)}.
To prove~{(b)}, suppose~$X\leq F_{\pi}(X)$ or~$X\geq F_{\pi}(X)$, then~{(a)} gives rise to
the diagram
\begin{equation*}
\xymatrix{
	X \ar[d]\ar[r] & F_{\pi}(X)\ar[d]\ar[r] &\cdots  \ar[r] & F_{\pi}^h(X)\ar[d]\\
	F_{\pi}(X)\ar[r] & F^2_{\pi}(X)        \ar[r] & \cdots  \ar[r] & F_{\pi}^{h+1}(X)
}
\end{equation*}
from which we derive the chain~$X, F_{\pi}(X), F_{\pi}^2(X),\ldots$ in the
poset~$\mathbb{F}_2^n$. Since~$\mathbb{F}_2^n$ is a finite poset, the chain becomes stationary after
a finite number of steps, that is,~$X$ reaches a fixed point.  Note if~$X\leq F^k_{\pi}(X)$ or~$X\geq F^k_{\pi}(X)$, then~$X$ reaches a fixed point w.r.t. iteration of the map $F^k_{\pi}$ according to~{(b)}. Thus, any two different states on a limit cycle of $F_{\pi}$ cannot be comparable, whence Part~{(c)}.  \qed

\begin{proposition}\label{4thm2}
  Let~$(G,f,\pi)$ be a monotone \SDS. Then the following statements hold.\\
  \phantom\quad (a) Suppose that the state~$X$ satisfies $X\leq Z$ or $X\geq Z$ and reaches the
  fixed point~$Z$.  Then any state~$Y$ satisfying~$X\leq Y \leq Z$ or~$X\geq Y\geq Z$ reaches the
  fixed point~$Z$.\\
  \phantom\quad (b) Suppose~$X$ satisfies~$X\leq Z_{\pi}$ or~$X\geq Z_{\pi}$ and reaches~$Z_{\pi}$.
  If~$Z_\pi$ is a fixed point of $(G,f,\pi)$, then for any~$\sigma$-schedule and~$k\geq 0$ we
  have~$F_{\sigma}^k(X)\leq Z_{\pi}$ or~$F_{\sigma}^k(X)\geq Z_{\pi}$.
\end{proposition}
\proof
To prove~(a) we may without loss of generality assume that~$X\leq Z$.  Lemma~\ref{4lem1} guarantees
that~$X\leq Y \leq Z$ implies that, for any $k\geq 0$, $F_{\pi}^k(X)\leq F_{\pi}^k(Y)\leq
F_{\pi}^k(Z)=Z$, where $F_{\pi}^0=\text{\rm id}$.  Since~$X$ reaches~$Z$, there exists~$k_0$ such
that~$F_{\pi}^k(X)=Z$ for any $k\geq k_0$. Consequently,~$Y$ reaches~$Z$ and~{(a)} follows.
As for~{(b)}, we note, that if~$Z_{\pi}$ is a fixed point for~$(G,f,\pi)$, then~$Z_{\pi}$ is a fixed
point for~$(G,f,\sigma)$ for any~$\sigma$, and we therefore have~$ F_{\sigma}^k(X)\leq
F_{\sigma}^k(Z_{\pi})=Z_{\pi}$, completing the proof.  \qed

The paper~\cite{goe1} considers the states~${\bf 0}$ and~${\bf 1}$ in order to probe for fixed
points in monotone systems and in addition discusses the existence of \GOE. From our previous
observations we can immediately conclude:
\begin{proposition}\label{observation}
  Let $(G,f,\pi)$ be a monotone \SDS. Then ${\bf 0}$ and ${\bf 1}$ are either a \GOE state reaching
  a fixed point or a fixed point.
\end{proposition}
\proof For any update schedule~$\pi$ it is clear that~$F_{\pi} ({\bf 0})\geq {\bf
  0}$. Lemma~\ref{4lem1} guarantees that if~${\bf 0}$ is a \GOE state, it will reach a fixed point.
Otherwise, there exists~$X> {\bf 0}$ and~$F_{\pi}(X)={\bf 0}$. Again, by Lemma~\ref{4lem1},~${\bf
  0}$ must be a fixed point since the formed decreasing chain $X>0\geq \cdots$ must become
stationary at~${\bf 0}$.  The argument for the case of~${\bf 1}$ is analogous and the proposition
follows.\qed

For a finite system, all fixed points form a sub-poset of the poset~$\mathbb{F}_2^n$.
In view of Proposition~\ref{4thm2} and Proposition~\ref{observation}, we obtain the following variation
of the Knaster-Tarski theorem~\cite{knaster, tarski}: the set of fixed points of a monotone function on
a complete lattice is a complete lattice. To this end, it suffices to establish the existence of a unique
maximal and minimal element (by definition of complete lattice).
\begin{proposition}\label{prop:knaster-tarski}\hskip5pt\phantom{Trick}\\
\phantom\quad (a) The set of fixed points of a monotone \SDS $(G,f,\pi)$ is a complete
lattice~$\mathbb{L}_\pi$.\\
\phantom\quad (b) In~$\mathbb{F}_2^n$, any periodic point~$X$ is comparable to
the~$\mathbb{L}_\pi$-maximum (MAX) and~$\mathbb{L}_\pi$-minimum (MIN), and we have~$\text{\rm
  MIN}\le X\le \text{\rm MAX}$.
\end{proposition}
\proof If~${\bf 1}$ is a fixed point, then it is certainly the unique maximum. Otherwise, suppose
that in the poset of fixed points the maximal elements are $Z_1,\ldots, Z_p$ where $p>1$. Since the
element~${\bf 1}$ is not a fixed point, Proposition~\ref{observation} guarantees that~${\bf 1}$
reaches a fixed point.

{\it Claim~$1$.} The fixed point reached by~${\bf 1}$ is maximal. \\ Suppose~${\bf 1}$ reaches the
fixed point $Z_0$.  If $Z_0$ is not comparable with any $Z_i$ for $i\neq 0$, then $Z_0$ is contained
in the set of maximal elements.  Otherwise we have~$Z_0<Z_i< {\bf 1}$ for some~$1\leq i \leq
p$. From Proposition~\ref{4thm2},~$Z_i$ reaches the fixed point $Z_0$ which implies that~$Z_i$
itself is not a fixed point, a contradiction, proving the claim.

Suppose~${\bf 1}$ reaches the fixed point~$Z_1$. Since~${\bf 1}>Z_i$ ($1< i \leq p$), monotonicity
implies~$F_{\pi}^k({\bf 1})\geq F_{\pi}^k(Z_i)=Z_i$ for $k\geq 1$. As~${\bf 1}$ reaches the fixed
point~$Z_1$, we arrive at~$Z_1\geq Z_i$, which is impossible since~$Z_1$ and~$Z_i$ are distinct,
maximal elements. Accordingly,~$Z_1$ is the unique maximal element.
The argument in case of~${\bf 0}$ is completely analogous.

Next, suppose that~$X$ is a periodic point. Then~$F_{\pi}^{mk}(X)=X$ for some~$k>0$ and any~$m>0$.
Since~${\bf 0}\leq X \leq {\bf 1}$, Lemma~\ref{4lem1} implies that~$F_{\pi}^l({\bf 0})\leq
F_{\pi}^l(X)\leq F_{\pi}^l({\bf 1})$ for any~$l\geq 0$. Since~${\bf 0}$ and~${\bf 1}$ reach MIN and
MAX as fixed points, we have, for a sufficiently large $m>0$
\begin{equation*}
{\rm MIN} = F_{\pi}^{mk}({\bf 0})\leq F_{\pi}^{mk}(X)=X \leq F_{\pi}^{mk}({\bf 1})={\rm MAX} \;,
\end{equation*}
completing the proof. \qed

\begin{remark}
  Proposition~\ref{prop:knaster-tarski} (b) appears to have not been addressed before and does not
  hold in the general case of systems with infinite phase space. The above proposition also implies
  that the phase space of a monotone system has specific properties that do not depend on the
  particular choice of local functions: if there exists a non-trivial periodic point, then there
  exist at least two fixed points.
\end{remark}
We denote the symmetric group on the set~$[n]=\{1,2,\ldots, n\}$ by~$\mathbb{S}_n$. Let $g\in
\mathbb{S}_n$ and $X=[x_1,x_2,\ldots, x_n]\in \mathbb{F}_2^n$. The group~$\mathbb{S}_n$ acts
on~$\mathbb{F}_2^n$ by
\begin{equation*}
  g \cdot X=[x_{g^{-1}(1)}, x_{g^{-1}(2)}, \ldots, x_{g^{-1}(n)}] \;.
\end{equation*}
Moreover, $\mathbb{S}_n$ acts on the set of update schedules via $g\cdot \pi=
\pi_{g^{-1}(1)}\pi_{g^{-1}(2)}\cdots \pi_{g^{-1}(n)}$.

Two \SDS phase spaces are called cycle equivalent~\cite{macmor} iff there exists an isomorphism
between their sets of limit cycles. Let~$\tau=(n \bai n-1 \bai \cdots \bai 1)$
be a cyclic permutation. For $\pi=\pi_1\pi_2\cdots \pi_n$ we set
\begin{equation*}
  \pi_{\tau^k}=\tau^k \cdot \pi=\pi_{k+1}\cdots \pi_n\pi_1\pi_2\cdots \pi_k \;.
\end{equation*}
There is a relation between the phase space of the~$\pi$-system~$F_{\pi}=F_{\pi_n}\circ
F_{\pi_{n-1}}\circ \cdots \circ F_{\pi_1}$ and
the~$\pi_{\tau^k}$-system~$F_{\pi_{\tau^k}}=F_{\pi_k}\circ F_{\pi_{k-1}} \circ \cdots \circ
F_{\pi_{k+1}}$:
\begin{proposition}\label{p-hom} The map
  $h=F_{\pi_k}\circ F_{\pi_{k-1}}\circ \dots \circ F_{\pi_1}$ is a homomorphism from the phase space
  of~$(G,f,\pi)$ to the phase space of~$(G,f,\pi_{\tau^k})$. Furthermore, restricted to the limit
  cycles~$h$ induces an isomorphism.
\end{proposition}
\proof It is sufficient to consider the case $k=1$, and we shall prove that $h=F_{\pi_1}$.

{\em Claim~$1$}. If $F_{\pi}(X)=Y$, then $F_{\pi_{\tau}}\big(F_{\pi_1}(X)\big)=F_{\pi_1}(Y)$.\\
We compute
\begin{align*}
  F_{\pi_{\tau}}\big(F_{\pi_1}(X)\big)&=F_{\pi_1}\circ F_{\pi_n}\circ F_{\pi_{n-1}} \circ \dots \circ F_{\pi_2}
  \big(F_{\pi_1}(X)\big)\\
&=F_{\pi_1}\big(F_{\pi_n}\circ F_{\pi_{n-1}} \circ \dots \circ F_{\pi_2}\circ F_{\pi_1}(X)\big)=F_{\pi_1}(Y).
\end{align*}
Thus, $h=F_{\pi_1}$ is a homomorphism as it maps the directed edge $X\rightarrow Y$ into the directed edge
$F_{\pi_1}(X)\rightarrow F_{\pi_1}(Y)$. It remains to prove that $h$ induces an isomorphism on limit cycles.
To prove this it is crucial that $F_{\pi}(X)\neq F_{\pi}(Y)$ implies $F_{\pi_1}(X) \neq F_{\pi_1}(Y)$. We conclude
from this: suppose for a sequence $(X_1,X_2,\dots, X_k)$, we have $F_{\pi}(X_i)=X_{i+1}$ for $1\leq i \leq k-1$
and $X_i\neq X_{j}$ for any $2\leq i<j \leq k$, then for the sequence
$$
\big(F_{\pi_1}(X_1),F_{\pi_1}(X_2),\dots, F_{\pi_1}(X_k)\big),
$$
we have $F_{\pi_1}(X_i)\neq F_{\pi_1}(X_{j}) $ for $2\leq i<j \leq k$ and
$F_{\pi_{\tau}}\big(F_{\pi_1}(X_i)\big)=F_{\pi_1}(X_{i+1})$ for $1\leq i\leq k-1$. Clearly, $X_k=X_1$ implies
$F_{\pi_1}(X_1)=F_{\pi_1}(X_k)$. As a result, $h$ preserves both: directed paths and limit cycles. Thus, each limit cycle of $(G,f,\pi)$ has a unique isomorphic copy under $(G,f,\pi_{\tau})$. Let $Cyc(F_{\pi})$ denote the set consisting of limit cycles of $F_{\pi}$. Note that $\pi=[\pi_{\tau}]_{\tau^{n-1}}$. Then we have the following diagram
\begin{equation*}
\xymatrix{
	& \pi \ar[d]  \ar[r]^{\tau}  & {\pi_{\tau}} \ar[d]\ar[r]^{\tau^{n-1}}  & \pi \ar[d]\\
	& Cyc(F_{\pi})\ar[d]  \ar[r]^{h_{\tau}}  & Cyc(F_{\pi_{\tau}})\ar[r]^{h_{\tau^{n-1}}}  & Cyc(F_{\pi}) \ar[d]\\
	& Cyc(F_{\pi})\ar@{-}[r]&^{F_{\pi}} \ar[r]   & Cyc(F_{\pi})
}
\end{equation*}
which implies that 
$(G,f,\pi)$ and $(G,f,\pi_{\tau})$ have the same number of limit cycles.
In particular, restricted to the limit cycles, $h$ induces an isomorphism.
\qed

As an immediate application of Proposition~\ref{p-hom}, we can recover the cycle equivalence result
of \SDS in~\cite{macmor} which was proved differently there.

\begin{corollary}\cite{macmor}
  The \SDS~$(G,f,\pi)$ and~$(G,f,\pi_{\tau^k})$ are cycle-equivalent.
\end{corollary}
Let~$\pi=\pi_1\pi_2\cdots \pi_n$ and~$\pi'=\pi'_1 \pi'_2 \cdots \pi'_n$ be two update
schedules. Let~$\pi\sim_{\alpha} \pi'$ if there exists some $k$ such that (i)~$\pi_j=\pi'_j$,~$j\neq
k, k+1$, and (ii) the vertices $\pi_k$ and $\pi_{k+1}$ are not adjacent in $G$. The transitive and
reflexive closure of~$\sim_{\alpha}$ gives an equivalence relation on the set of all update
schedules~\cite{rei1}.  We denote the equivalence class of~$\pi$ by~$[\pi]_{\alpha}$, and for any
update schedules~$\pi$ and~$\pi'$ for which~$\pi\sim_{\alpha} \pi'$, we have~$F_{\pi}=F_{\pi'}$ by
construction.

For $X=[x_1,x_2,\ldots, x_n]$ and~$\pi=\pi_1\pi_2\cdots \pi_n$, let
\begin{equation*}
	[X]_{\pi}=\{\sigma \cdot X \mid \sigma\in [\pi]_{\alpha}\} \;,
\end{equation*}
where we consider update schedules as permutations using one-line representation.
Let~$S_{0,k}=[0,0,\dots,0,1,1,\dots, 1]\in \mathbb{F}_2^n$ where~$x_i=0$ for~$i\le k$ and~$x_j=1$
otherwise, and define~$S_{1,k}=[1,1,\dots,1,0,0,\dots, 0]$ analogously.
\begin{theorem}\label{t-main}
	Let~$(G,f,\pi)$ be a monotone \SDS with update schedule~$\pi$.
	Then any state~$X\in [S_{0,k}]_{\pi}\bigcup [S_{1,k}]_{\pi}$ for some~$k$ with~$1\leq k \leq n$ is either a \GOE state or reaches a
	fixed point.  
\end{theorem}
\begin{proof}
	Without loss of generality we may assume that~$X\in [S_{0,k}]_{\pi}$. We shall prove that
	if~$X$ is not a \GOE state, then~$X$ reaches a fixed point.
	First, if~$X\in [S_{0,k}]_{\pi}$, then there exists an update schedule~$\sigma\in
	[\pi]_{\alpha}$ such that~$\sigma^{-1} \cdot X=S_{0,k}$. This can be seen as follows: if for $\sigma \in [\pi]_{\alpha}$, $X=\sigma \cdot S_{0,k}$, since it is a group action, we have
	$\sigma^{-1} \cdot X=\sigma^{-1} \cdot (\sigma \cdot S_{0,k})=S_{0,k}$.
	
	If~$X$ reaches a fixed point
	of~$(G,f,\sigma)$, then~$X$ also reaches a fixed point in $(G,f,\pi)$ since~$F_{\pi}=F_{\sigma}$.
	By relabeling the vertices~$\sigma_i$ by~$i$ ($1\leq i \leq n$), it is sufficient to show
	that~$X=S_{0,k}$ reaches a fixed point of~$(G,f,{\rm id})$.
	
	If~$X=S_{0,k}$ is a fixed point we are done. If~$X$ is neither a \GOE state nor a fixed point, then there
	exist states~$Y$ and~$Z$ such that~$F_{\pi}(Y)=X$ and~$F_{\pi}(X)=Z\neq X$.
	Let~$\pi'={(k+1)}{(k+2)}\cdots {n}1\cdots k$. Then we have
	\begin{align*}
		F_{\pi'}([x_1,\dots, x_k, y_{k+1},\dots, y_n])&=F_{\pi'}(F_k\circ F_{k-1}\circ \cdots \circ F_1(Y))\\
		&=F_k\circ F_{k-1}\circ \cdots \circ F_1\circ F_{\pi}(Y)\\
		&=F_k\circ F_{k-1}\circ \cdots \circ F_1(X)=
		[z_1,\dots, z_k, x_{k+1},\dots, x_n].
	\end{align*}
	By assumption, we have~$X=S_{0,k}$, that is,~$x_i=0$ for~$1\leq i \leq k$ and~$x_i=1$ for~$k+1\leq i
	\leq n$.  As a result, for any~$y_i$ ($k+1\leq i \leq n $) and any~$z_i$ ($1\leq i \leq k$) we have:
	\begin{equation*}
		[x_1,\dots, x_k, y_{k+1},\dots, y_n]\leq [z_1,\dots, z_k, x_{k+1},\dots, x_n].
	\end{equation*}
	Lemma~\ref{4lem1} gives rise to a chain, and under $(G,f,\pi')$ the states~$[x_1,\dots, x_k,
	y_{k+1},\dots, y_n]$ as well as~$[z_1,\dots, z_k, x_{k+1},\dots, x_n]$ reach the same fixed
	point. Proposition~\ref{p-hom} shows that~$h=F_k\circ \dots \circ F_1$ is a homomorphism from the
	phase space of~$(G,f,\pi)$ to that of~$(G,f,\pi')$ and, furthermore, that $h(X)=[z_1,\dots, z_k,
	x_{k+1},\dots, x_n]$.
	Suppose~$X$ reaches a nontrivial limit cycle of~$(G,f,\pi)$. Since~$h$ restricted to limit cycles
	is an isomorphism, the~$h$-image of this limit cycle remains nontrivial, whence~$[z_1,\dots, z_k,
	x_{k+1},\dots, x_n]$ reaches a nontrivial limit cycle in~$(G,f,\pi')$, which is
	impossible. Accordingly,~$X$ reaches a fixed point of~$(G,f,\pi)$, completing the proof.
\end{proof}

\begin{remark}
	Theorem~\ref{t-main} can be generalized from the vertex state set being~$\mathbb{F}_2$ to any
	poset~$P$ as long as~$P$ has a minimum and a maximum.
\end{remark}
\begin{corollary}\label{coro-main-1}
	Let~$(G,f,\pi)$ be a monotone \SDS. If~$X\in [S_{0,k}]_{\pi}\bigcup [S_{1,k}]_{\pi}$
	and~$F_{\pi}(X)$ is not comparable to~$X$, then~$X$ is a \GOE.  If~$X\in [S_{0,k}]_{\pi}$
	and~$X>F_{\pi}(X)$, then $X$ is a \GOE state. If~$X\in [S_{1,k}]_{\pi}$ and~$X<F_{\pi}(X)$, then~$X$
	is a \GOE state.
\end{corollary}
\begin{proof} 
	Suppose $X\in [S_{0,k}]_{\pi}$. It suffices to consider $X=S_{0,k}$ under~$(G,f,{\rm id})$ by relabeling. Suppose $F_{\pi}(Y)=X$ and~$F_{\pi}(X)=Z\neq X$.
	Note from the phase space of $(G,f,\pi_{\tau^k})$ to the phase space of $(G,f,\pi)$,
	$h=F_n\circ \cdots \circ F_{k+1}$ gives the homomorphism according to Proposition~\ref{p-hom}. It can be checked that $h([x_1,\dots, x_k, y_{k+1},\dots, y_n])=X$
	and $h([z_1,\dots, z_k, x_{k+1},\dots, x_n])=Z$. By assumption,  $h$ as compositions of monotone functions is monotone and $[x_1,\dots, x_k, y_{k+1},\dots, y_n]< [z_1,\dots, z_k, x_{k+1},\dots, x_n]$. Therefore, $X<Z=F_{\pi}(X)$. Hence, if $X$ is not a \GOE state and $F_{\pi}(X)\neq X$, then $X<F_{\pi}(X)$, which implies the corollary. 
\end{proof}

\GOE states have been analyzed extensively~\cite{goe1,moore,myhill}, see for instance the Garden of
Eden theorem of Moore and Myhill in the context of (infinite) cellular
automata~\cite{moore,myhill}. Determining if a particular state is a \GOE state is generally
hard~\cite{prx}. Given an update schedule, Theorem~\ref{t-main} and Corollary~\ref{coro-main-1}
allow one to identify states which are either \GOE states or reach a fixed point.  It is worth
pointing out that the framework presented facilitates identification of fixed points in monotone
dynamical systems where vertices are not updated sequentially since fixed points do not depend on
the order of the updates. 

\begin{example}
	\begin{figure}[!htb]
		\centering
		\includegraphics[width=0.4\textwidth]{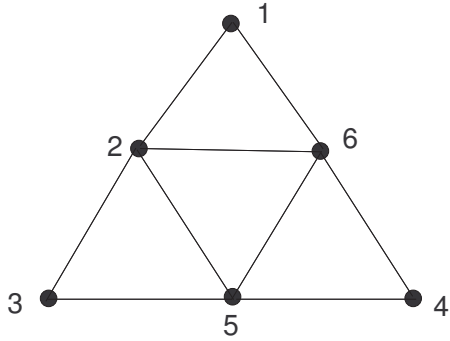}
		\caption{A graph $G$ with $6$ vertices.}
		\label{fig:thresh-ex}
	\end{figure}
	Considers Figure~\ref{fig:thresh-ex} and let~$\pi=241635$. It is easy to check that
	\begin{multline*}
		[\pi]_{\alpha}=\{241635,241365,  243165, 214635,  214365,\\
		213465, 234165, 231465, 421635, 421365,423165\} \;,\phantom{This will look a little nicer?}
	\end{multline*}
	
	and we obtain
	\begin{align*}
		[S_{0,3}]_{\pi}=\{
		001011, 100011,000111\} \;.
	\end{align*}
	Suppose all local functions are simple threshold
	functions, where the threshold values for the vertices
	$1$,~$2$,~$3$,~$4$,~$5$, and~$6$ are
	$1$,~$2$,~$1$,~$2$,~$2$, and~$3$, respectively.
	Then for~$X=[100011]\in [S_{0,3}]_{\pi}$ we have~$F_{\pi}(X)=[111111]$ which is a fixed
	point, that is,~$X$ reaches a fixed point. Next, suppose the threshold values are
	$2$,~$4$,~$1$,~$2$,~$2$, and~$3$,
	respectively.
	Then we have~$F_{\pi} (X)=[001111]$ which is incomparable with $X$, and we can conclude that in view of Corollary~\ref{coro-main-1} $X$
	is a \GOE state.
\end{example}
For a state~$X$ and~$i\in \{0,1\}$ we set~$\vartheta_{i}(X)=\{\pi \mid \pi^{-1}\cdot X=S_{i,k}\quad
\mbox{for some $k$}\}$.
\begin{lemma}\label{lem:3mainlem}
  Suppose~$X$ is a non-trivial periodic point in the phase space~$\mathbb{G}$ of some dynamical
  system.  Then, for any~$\pi\in \vartheta_0(X)\bigcup \vartheta_1(X)$, no monotone \SDS~$(G,f,\pi)$
  can generate~$\mathbb{G}$.
\end{lemma}
\proof Suppose $\mathbb{G}$ is the phase space of a monotone \SDS~$(G,f,\pi)$ with~$\pi\in
\vartheta_0(X)\bigcup \vartheta_1(X)$. According to Theorem~\ref{t-main}, $X$ is either a \GOE state
or reaches a fixed point of~$(G,f,\pi)$, which implies that $X$ is not a non-trivial periodic
point. \qed

Lemma~\ref{lem:3mainlem} gives a sufficient condition for a phase space to not be generated by some
monotone \SDS. Namely, the existence of a subset of states $A$ such that all states in $A$ are
non-trivial periodic states and $\bigcup_{X\in A} \vartheta_i(X) =\mathbb{S}_n$.
\begin{lemma}\label{lem:nonintersect}
If two states~$X$ and~$Y$ are not comparable, then $\vartheta_0(X)\bigcap \vartheta_0(Y)=\varnothing$.
\end{lemma}
\proof Suppose there exists~$\pi=\pi_1\pi_2\cdots \pi_n\in \vartheta_0(X)\bigcap \vartheta_0(Y)$.
We have $\pi^{-1} \cdot Z=S_{0,k}$ if and only if $z_{\pi_i}=0$ for $1\leq i\leq k$ and
$z_{\pi_i}=1$ otherwise. By assumption,~$\pi^{-1} \cdot X = S_{0,k_X}$ and $\pi^{-1} \cdot
Y=S_{0,k_Y}$, and since $k_X\le k_Y$ or $k_Y\le k_X$ it follows that~$X$ and~$Y$ are comparable,
which is impossible.\qed

Lemma~\ref{lem:nonintersect} immediately implies that if~$A$ is a set of states which are mutually
incomparable, then
\begin{equation*}
  \sum_{X\in A} |\vartheta_0(X)| = \left|\bigcup_{X\in A} \vartheta_0(X)\right| \leq n! \;.
\end{equation*}
An immediate consequence of this is the celebrated LYM inequality~\cite{lubell} and Sperner's lemma which estimate
the sizes of incomparable sets within the power set of a finite set of size~$n$.
\begin{proposition}\label{P:uncomp}
In~$(\mathbb{F}_2^n,\le)$ let~$A\subset \mathbb{F}_2^n$ be a set of states whose elements are
mutually incomparable.  Suppose there are exactly~$a_k$ states of~$A$ having~$k$ coordinates for
which~$x_i=0$. Then we have
\begin{equation*}
	\sum_k \frac{a_k}{{n\choose k}} \leq 1 \;.
\end{equation*}
In particular, we have $\mid A\mid \leq {n\choose \lfloor \frac{n}{2}\rfloor}$.
\end{proposition}
\proof For $X\in A$ having $k$ coordinates for which~$x_i=0$, it is easy to see that
$|\vartheta_0(X)|=k!(n-k)!$.  Then
\begin{equation*}
\sum_{X\in A } |\vartheta_0(X)|=\sum_k a_k k!(n-k)!\leq n! \;,
\end{equation*}
which produces the LYM inequality. Using the fact that~${n\choose k}\leq {n\choose \lfloor
  \frac{n}{2}\rfloor}$ we obtain Sperner's lemma
\begin{equation*}
|A|=\sum_{k}a_k \leq {n\choose \lfloor \frac{n}{2}\rfloor}.
\end{equation*}
\qed

The proof in Lubell~\cite{lubell} is essentially the same as the one given here, but we have a
different motivation of relating states (or subsets) to permutations (i.e., update schedules).  Note
that if~$k\neq \lfloor \frac{n}{2}\rfloor$ and~$k\neq n-\lfloor \frac{n}{2}\rfloor$, then~${n\choose
  k}< {n\choose \lfloor \frac{n}{2}\rfloor}= {n\choose n-\lfloor \frac{n}{2}\rfloor}$.  Hence,~$\mid
A\mid = {n\choose \lfloor \frac{n}{2}\rfloor}$ only if~$A$ is either the set of states having
exactly~${ \lfloor \frac{n}{2}\rfloor}$ coordinates, where~$x_i=1$ or the set having exactly~$n-{
  \lfloor \frac{n}{2}\rfloor}$ such coordinates.

According to Sperner's lemma, the maximum possible size of a limit cycle (as an anti-chain) of a
monotone system is ${n\choose \lfloor \frac{n}{2}\rfloor}$. Our next theorem shows that this maximum
is not achievable for monotone \SDS.

\begin{theorem}\label{thm:maxlength}
 Let~$(G,f,\pi)$ be a monotone \SDS. The size of any limit cycle of~$(G,f,\pi)$ is strictly less
 than~${n \choose \lfloor n/2 \rfloor}$.
\end{theorem}
\proof We prove by contradiction. Suppose a monotone \SDS has a limit cycle of length~${n \choose
  \lfloor n/2 \rfloor}$. Since a limit cycle of a monotone system gives an anti-chain, the limit
cycle of length~${n \choose \lfloor n/2 \rfloor}$ must be either the set~$A_1$ of states having
exactly~${ \lfloor \frac{n}{2}\rfloor}$ coordinates where~$x_i=1$, or the set~$A_2$ of states having
exactly~$n-{ \lfloor \frac{n}{2}\rfloor}$ such coordinates.  According to
Lemma~\ref{lem:nonintersect}, we have
\begin{equation*}
 \left|\bigcup_{X\in A_i} \vartheta_0(X)\right|=\sum_{X\in A_i} |\vartheta_0(X)|=
 \sum_{X\in A_i} { \lfloor \frac{n}{2}\rfloor}! \left(n-{ \lfloor \frac{n}{2}\rfloor}\right)!=
     {n \choose \lfloor n/2 \rfloor}  { \lfloor \frac{n}{2}\rfloor}! \left(n-{ \lfloor \frac{n}{2}\rfloor}\right)!=n!,
\end{equation*}
for any $i\in \{1,2\}$, that is, $\bigcup_{X\in A_i} \vartheta_0(X)$ contains all update schedules.
Without loss of generality, we may assume that~$A_1$ is the limit cycle of~$(G,f,\pi)$.
Lemma~\ref{lem:3mainlem} shows that a phase space with states in~$A_1$ being non-trivial periodic
points can not be realized by any monotone \SDS~$(G',f',\sigma)$ for any~$\sigma\in \bigcup_{X\in
  A_1} \vartheta_0(X)$, whence the theorem. \qed

We proceed by presenting some implications of Theorem~\ref{t-main}:
\begin{corollary}\label{prob}
  Let~$G$ be a graph on~$n>1$ vertices with fixed local, monotone functions~$f_i$. Selecting a
  permutation~$\pi$ induces the \SDS~$(G,f,\pi)$, and the probability that under a random~$\pi$
  having~$X$ as either a \GOE state or reaching a fixed point is at least~$\frac{2}{{n \choose
      {\lfloor n/2\rfloor}}}$.  Furthermore, the probability of a random state being either a \GOE
  state or reaching a fixed point under a random update schedule is at least~$\frac{n}{2^{n-1}}$.
\end{corollary}
\proof Let~$X$ be some fixed state with~$m$ coordinates such that~$x_j=0$ and $n-m$ coordinates for
which~$x_j=1$ where $0<m<n$.  Considering~$\pi^{-1} \cdot X=S_{0,m}$ and~$\pi^{-1} \cdot
X=S_{1,n-m}$, we conclude that there are~$m!(n-m)!$ different update schedules such that~$X$ is of
the form~$S_{0,m}$ and~$m!(n-m)!$ different update schedules for which~$X$ is of the
form~$S_{1,n-m}$.  By Theorem~\ref{t-main}, for each such permutation,~$X$ is either a \GOE state
or reaches a fixed point.  Thus we obtain the first probability to be at least
\begin{equation*}
\frac{2m!(n-m)!}{n!}=\frac{2}{{n\choose m}}\geq \frac{2}{{n\choose {\lfloor n/2 \rfloor}}} \;,
\end{equation*}
where the last inequality follows from~${n\choose m}\leq {n\choose {\lfloor n/2 \rfloor}}$.
Proposition~\ref{observation} guarantees that in case~$X={\bf 0}$ or~$X={\bf 1}$, $X$ is either a
\GOE state or reaches a fixed point for any update schedule, whence in this case the probability
is~$1$.

The total number of pairs~$(X,\pi)$, where~$X$ is either a \GOE state or reaches a fixed point
of~$(G,f,\pi)$ is at least~$\sum_{0<m<n}{n\choose m}2m!(n-m)!+2n!=2n\cdot n!$. Thus, the second
probability in question is at least~$\frac{2n\cdot n!}{2^n n!}=\frac{n}{2^{n-1}}$, completing the
proof of the corollary. \qed
\begin{corollary}
Let~$G$ be a bipartite graph with vertex set~$U\dot\bigcup V$ where~$U=\{u_1,u_2,\dots, u_m\}$
and~$V=\{v_1,v_2,\dots, v_n\}$.  Let $\pi=u_1\cdots u_m v_1\cdots v_n$ and~$(G,f,\pi)$ be a monotone
\SDS.  Then, the probability of a random state being either a \GOE state or reaching a fixed point
in~$(G,f,\pi)$ is at least~$\frac{2^{m+1}+2^{n+1}-2^2}{2^{m+n}}$.
\end{corollary}
\proof Suppose~$X=[x_{u1},x_{u2},\ldots, x_{um},x_{v1}, x_{v2}, \ldots, x_{vn}]$. For any update
schedule~$\sigma$ of the form~$\sigma=u_{k_1} u_{k_2}\cdots u_{k_m} v_{l_1} v_{l_2} \cdots v_{l_n}$,
we have~$\sigma \in [\pi]_{\alpha}$, since no pair of $u$-vertices or $v$-vertices are adjacent
in~$G$.

In particular, any state~$X$ for which $x_{u_i}=1$ is contained in~$[S_{1,d}]_{\pi}$ where~$d\ge m$,
and any state~$X$ where~$x_{u_i}=0$ is contained in~$[S_{0,d}]_{\pi}$ where $d\ge m$.  Any state~$X$
where~$x_{v_i}=1$ is clearly in~$[S_{0,j}]_{\pi}$, $j\le n$ and any state~$X$ where~$x_{vi}=0$ is
contained in~$[S_{1,j}]_\pi$, $j\le n$.

By construction, note that the states~${\bf 0}$,~${\bf 1}$, the state~$x_{u_i}=1$ while~$x_{v_j}=0$,
and the state~$x_{u_i}=0$ while~$x_{v_j}=1$ are counted twice. Thus, there are at
least~$2^{m+1}+2^{n+1}-2^2$ states each of which is either a \GOE state or reaches a fixed point
of~$(G,f,\pi)$.  Consequently, the probability in question is at
least~$\frac{2^{m+1}+2^{n+1}-2^2}{2^{m+n}}$ and the corollary follows.  \qed

\section{Sequentializing Monotone Parallel Systems}\label{sec4}
Let~$G$ be a simple graph with vertex set~$V(G)=\{1,2,\ldots,n\}$ where each vertex~$i$ has a
binary state and a monotone, local function~$f_i$. Updating all vertex states in parallel produces
the parallel monotone dynamical system of~$G$ and~$f$ denoted by~$(G,f)$.

For certain systems, it might not be possible to maintain accurate synchronization of all vertices
in the systems as required under a parallel update. In such cases, a sequential update, possibly
over a different graph $G'$ and monotone, local functions, $f_i'$ generating the same dynamics as
the parallel system may be desirable.
\begin{definition}
Let~$(G,f)$ be a parallel dynamical system and~$(G',f',\pi)$ an \SDS where $V(G)=V(G')$.
Then~$(G',f',\pi)$ is a sequentialization of~$(G,f)$ and~$(G,f)$ is a parallelization
of~$(G',f',\pi)$ iff~$(G',f',\pi)=(G,f)$.
\end{definition}
Any \SDS has a parallelization. Namely, given an \SDS, $(G,f,\pi)$, we may assume without lost of
generality that $\pi=12\cdots n$ with the underlying local maps:
\begin{align*}
x_1 \mapsto y_1 &= f_1([x_1,x_2,x_3,\ldots, x_n]),\\
x_2 \mapsto y_2 &= f_2([y_1,x_2,x_3,\ldots, x_n]),\\
x_3 \mapsto y_3 &= f_3([y_1,y_2,x_3,\ldots, x_n]),\\
& \, \, \, \vdots\\
x_n \mapsto y_n &= f_n([y_1,y_2,y_3,\ldots, x_n]).
\end{align*}
Then the parallel system, $(G',f')$, whose local maps are given by
\begin{align}
x_1 \mapsto y_1 &=f'_1([x_1,x_2,x_3,\ldots, x_n])= f_1([x_1,x_2,x_3,\ldots, x_n]), \nonumber\\
x_2 \mapsto y_2 &=f'_2([x_1,x_2,x_3,\ldots, x_n])= f_2([y_1,x_2,x_3,\ldots, x_n]), \nonumber\\
x_3 \mapsto y_3 &=f'_3([x_1,x_2,x_3,\ldots, x_n])= f_3([y_1,y_2,x_3,\ldots, x_n]), \tag{$*$}\\
& \, \, \, \vdots \nonumber\\
x_n \mapsto y_n &= f'_n([x_1,x_2,x_3,\ldots, x_n])=f_n([y_1,y_2,y_3,\ldots, x_n]) \nonumber
\end{align}
and where $G'$ is implied by the dependencies of the $f'_i$ on the $x_j$ represents the
parallelization of $(G,f,\pi)$. If $f_1,\dots,f_n$ are monotone, for $X\le X'$, we have for any $i$,
$(y_1,\dots,y_{i-1},x_i,\dots,x_n)\le (y_1',\dots,y_{i-1}',x_i',\dots,x_n')$ and hence
$$
f'_i(X)=f_i(y_1,\dots,y_{i-1},x_i,\dots,x_n)\le f_i(y_1',\dots,y_{i-1}',x_i',\dots,x_n')=f'_i(X').
$$
As a result, if $(G,f,\pi)$ is monotone, then its
parallelization, $(G',f')$, is a monotone, parallel system.

We next analyze whether or not a parallel system~$(G',f')$ can be sequentialized. For particular
classes of systems, such a sequentialization is always possible, for instance, linear, parallel
systems can always be sequentialized as linear \SDS~\cite{chr}. In the following we shall show that
there exist monotone parallel systems for which there is no monotone sequentialization.

Let~$(G,f)$ be a monotone, parallel dynamical system. Lemma~\ref{lem:3mainlem} implies that if~$X$
is a non-trivial periodic point of $(G,f)$, then for any $\pi\in \vartheta_0(X)\bigcup
\vartheta_1(X)$, the \SDS~$(G',f',\pi)$ is not a sequentialization of~$(G,f)$.
\begin{theorem}
There exists a monotone parallel dynamical system which does not have a monotone sequentialization.
\end{theorem}
\proof By Theorem~\ref{thm:maxlength}, it is sufficient to find a monotone, parallel dynamical
system which has a limit cycle of length~${n\choose \lfloor \frac{n}{2}\rfloor}$.
In~\cite{goles-mon}, such a monotone, parallel system is constructed: let~$A=\{X_0,X_1,\ldots,
X_{p-1}\}$ be the set of states having exactly~${\lfloor \frac{n}{2}\rfloor}$
coordinates~$x_i=1$,~$p={n\choose \lfloor \frac{n}{2}\rfloor}$ and~$A^c=\mathbb{F}_2^n\setminus A$.
We construct a parallel dynamical system
\begin{equation*}
 (G,f) \colon X_i\mapsto X_{i+1} \;,
\end{equation*}
where the indices are taken modulo $p$,
and for $Y\in A^c$,
\begin{equation*}
(G,f) : Y\mapsto
  \begin{cases}
    {\bf 1}, & \text{if there exists $X\in A$ such that $X<Y$;}  \\
    {\bf 0}, & \text{if there exists $X\in A$ such that $X>Y$}.
  \end{cases}
\end{equation*}
By Proposition~\ref{P:uncomp}, this map is well defined and monotone, whence the theorem.\qed

In the following we shall further discuss the sequentialization of monotone, parallel
systems. Suppose the monotone parallel system~$(G',f')$ has a sequentialization~$(G,f,\pi)$
for~$\pi=12\cdots n$\footnote{Without loss of generality we may assume that $\pi=12\cdots
  n$ by relabeling.} as in~$(*)$, above.

Suppose for $[0,x_2,\ldots, x_n]<[1,x'_2,\ldots, x'_n]$ where $x_i<x'_i$ for some $2\leq i \leq n$, we have
\begin{align*}
  f_1([1,x_2,\ldots, x_n]) &= f'_1([1,x_2,\ldots, x_n]) =0,\\
  f_1([0,x'_2,\ldots, x'_n]) &= f'_1([0,x'_2,\ldots, x'_n]) =1 \;.
\end{align*}
As~$[1,x_2,\ldots, x_n]$ and~$[0,x'_2,\ldots, x'_n]$ are not comparable, such a monotone function
$f'_1$ exists.  By construction,~$[0,x_2,\ldots, x_n]\le [1,x_2,\ldots, x_n]$, and in view of $(*)$,
we have
\begin{equation*}
f_2([0,x_2,\ldots, x_n])=f'_2([1,x_2,\ldots, x_n]),
\qquad
f_2([1,x'_2,\ldots, x'_n])=f'_2([0,x'_2,\ldots, x'_n]) \;.
\end{equation*}
However, since~$[1,x_2,\ldots, x_n]$ and~$[0,x'_2,\ldots, x'_n]$ are not comparable, the
monotonicity of~$f_2'$ does not necessarily imply that~$f'_2([1,x_2,\ldots, x_n])\le
f'_2([0,x'_2,\ldots, x'_n])$, that is, we cannot conclude that~$f_2$ is monotone.  Accordingly,
monotonicity is not guaranteed, even if the underlying local maps of the parallel system are
monotone.
\begin{lemma}\label{mono-ext}
Suppose $g\colon A\rightarrow \mathbb{F}_2$ is monotone where $A\subset \mathbb{F}_2^n$. Then there
exists a monotone function $\hat{g} \colon \mathbb{F}_2^n \rightarrow \mathbb{F}_2$ such
that~$\hat{g}(X)=g(X)$ if $X\in A$.
\end{lemma}
\proof For~$X\in A$, we set~$\hat{g}(X)=g(X)$. We extend~$\hat{g}$ from~$A$ to~$\mathbb{F}_2^n$
inductively, using the following procedure:
\begin{description}
  \item[Step $1$.] Set $B=A$.
  \item[Step $2$.] Let $Y\in \mathbb{F}_2^n \setminus B$, and
                   let $Max=\max\{g(Z): (Z>Y) \wedge (Z\in B) \}$
                   and $Min=\max\{g(Z): (Z<Y) \wedge (Z\in B) \}$.
  \item[Step $3$.] Set $\hat{g}(Y)=Min$ and set $B=B\bigcup \{Y\}$.
                   If $B\neq \mathbb{F}_2^n$, then go to Step $2$.
\end{description}
The above procedure generates a monotone function~$\hat{g}\colon \mathbb{F}_2^n \rightarrow
\mathbb{F}_2$. \qed
\begin{proposition}
Suppose the monotone parallel system~$(G',f')$ has a sequentialization~$(G,f,\pi)$
where~$\pi=12\cdots n$. For $1\leq i \leq n$, let~$A_i=\{Z: Z=f'_i(X), X\in \mathbb{F}_2^n\}$.
Then~$(G',f')$ can be sequentialized as a monotone \SDS with respect to~$\pi$ if and only if~$f_i$
is monotone on~$A_i\subseteq \mathbb{F}_2^n$ for any~$1\leq i \leq n$.
\end{proposition}
\proof First, if~$(G,f,\pi)$ is a sequentialization of~$(G',f')$, then any
\SDS~$(\hat{G},\hat{f},\pi)$ is a sequentialization of~$(G',f')$ as long as~$\hat{f}_i$ and~$f_i$
agree on~$A_i$, that is, the dynamics does not depend on the behavior of~$f_i$
on~$\mathbb{F}_2^n\setminus A_i$.

Secondly, if~$f_i$ is monotone on~$A_i$ then, by Lemma~\ref{mono-ext}, there exists~$\hat{f}_i$,
which is monotone on~$\mathbb{F}_2^n$ and agrees with~$f_i$ on~$A_i$. Hence~$(G',f')$ can be
sequentialized as a monotone \SDS.

Finally, if~$(G',f')$ can be sequentialized via the monotone \SDS~$(\hat{G},\hat{f},\pi)$, both~$(G',f')$
and~$(G,f,\pi)$ as well as~$(G',f')$ and~$(\hat{G},\hat{f},\pi)$ satisfy~$(*)$. Comparing the two systems
of equations, we observe that~$\hat{f}_i$ and~$f_i$ agree on~$A_i$. By assumption,~$\hat{f}_i$ is monotone
on~$\mathbb{F}_2^n$, and thus~$f_i$ is monotone on~$A_i$.
\qed

\subsection*{Acknowledgments}
We thank Andrew Warren for discussions. The third author is a Thermo Fisher Scientic Fellow in Advanced Systems for Information Biology and acknowledges their support of this work.


\begin{thebibliography}{99}
	\bibitem{mon-neural2} J. Aracena, J. Demongeot, E. Goles, Positive and negative circuits in
	discrete neural networks, IEEE Trans. Neural Netw.~15(1) (2004) 77--83.
	
	\bibitem{aracena} J. Aracena, Maximum number of fixed points in regulatory boolean networks,
	Bull. Math. Biol.~70(5) (2008) 1398--1409.
	
	\bibitem{goles-mon} J. Aracena, J. Demongeot, E. Goles, On limit cycles of monotone functions with
	symmetric connection graph, Theor. Comput. Sci.~322 (2004) 237--244.
	
	\bibitem{rei1} C.L. Barrett, C.M. Reidys, Elements of a theory of simulation I,
	Appl. Math. Comput.~98 (1999) 241--259.
	
	\bibitem{rei2} C.L. Barrett, H.S. Mortveit, C.M. Reidys, Elements of a theory of simulation II:
	sequential dynamic systems, Appl. Math. Comput.~107 (2000) 121--136.
	
	\bibitem{rei3} C.L. Barrett, H.S. Mortveit, C.M. Reidys, Elements of a theory of simulation III:
	equivalence of SDS, Appl. Math. Comput.~122 (2001) 325--340.
	
	\bibitem{rei4} C.L. Barrett, H.S. Mortveit, C.M. Reidys, Elements of a theory of simulation IV:
	sequential dynamics systems: fixed points, invertibility and equivalence,
	Appl. Math. Comput.~134 (2003) 153--171.
	
	\bibitem{goe1} {C.L. Barrett, H.B. Hunt III, M.V. Marathe, S.S. Ravi, D.J. Rosenkrantz,
		R.E. Stearns, P.T. Tosic}, Gardens of Eden and Fixed Points in Sequential Dynamical Systems,
	DMTCS Proceedings AA (2001)~95--110.
	
	\bibitem{prx} C.L. Barrett, H.B. Hunt III, M.V. Marathe, S.S. Ravi,
	D.J. Rosenkrantz,. R.E. Stearns, M. Thakur, Predecessor existence problems for finite discrete
	dynamical systems, Theor. Comput. Sci.~386(1-2) (2007) 3--37.
	
	\bibitem{linear} W.Y.C. Chen, X. Li, M.J. Zheng, Matrix method for linear sequential dynamical
	systems on digraphs, Appl. Math. Comput.~160 (1) (2005) 197--212.
	
	\bibitem{chr} R.X.F. Chen, C.M. Reidys, Linear sequential dynamical systems, incidence algebras, and M{\"o}bius
	functions, Linear Algebra Appl., accepted.
	
	\bibitem{monomial} O. Col{\'o}n-Reyes, R. Laubenbacher, B. Pareigis, Boolean Monomial Dynamical
	Systems, Ann. Combinat.~8 (2004) 425--439.
	
	\bibitem{mon-neural1} H. Daniels, M. Velikova, Monotone and Partially Monotone Neural Networks,
	IEEE Trans. Neural Netw.~21 (2010) 906--917.
	
	\bibitem{and-or} E. Goles, G. Hern\'{a}ndez, Dynamical behavior of Kauffman networks with AND-OR
	gates, J. Biol. Syst.~8 (2000) 151--175.
	
	\bibitem{hopfield} J. Hopfield, Neural networks and physical systems with emergent collective
	computational abilities, Proc. Nat. Acad. Sc. U.S.A.~79 (1982) 2554--2558.
	
	\bibitem{laub} A.S. Jarrah, R. Laubenbacher, A. Veliz-Cuba, The dynamics of conjunctive and
	disjunctive boolean network models, Bull. Math. Biol.~72 (2010) 1425--1447.
	
	\bibitem{kauf} S.A. Kauffman, Metabolic stability and epigenesis in randomly constructed genetic
	nets, J. Theor. Biol.~22 (1969) 437--467.
	
	\bibitem{knaster} B. Knaster, A. Tarski, Un th\'{e}or\`{e}me sur les fonctions d'ensembles,
	Ann. Soc. Polon. Math.~6 (1928) 133--134.
	
	\bibitem{lubell} D. Lubell, A short proof of Sperner's lemma, ‎J. Comb. Theory 1(2) (1966) 299.
	
	\bibitem{moore} E.F. Moore, Machine models of self-reproduction, Proc. Symp. Appl. Math.~14 (1962)
	17--33.
	
	\bibitem{myhill} J. Myhill, The converse of Moore's Garden-of-Eden theorem,
	Proc. Am. Math. Soc.~14 (1963) 685--686,
	
	\bibitem{rei5} H.S. Mortveit, C.M. Reidys, An Introduction to Sequential Dynamic Systems,
	Springer,~2008.
	
	\bibitem{macmor} M. Macauley, H.S. Mortveit, Cycle equivalence of graph dynamical systems,
	Nonlinearity~22(1) (2009) 421--436.
	
	\bibitem{vonn} J. von Neumann, Theory of Self-Reproducing Automata, University of Illinois Press,
	Chicago,~1966.
	
	\bibitem{sontag1} E. Sontag, Monotone and near-monotone biochemical networks,
	J. Sys. Synth. Biol.~1 (2007) 59--87.
	
	\bibitem{sontag2} E. Sontag, A. Veliz-Cuba, R. Laubenbacher, A.S. Jarrah, The effect of negative
	feedback loops on the dynamics of Boolean networks, Biophys. J.~95 (2008) 518--526.
	
	\bibitem{tarski} A. Tarski, A lattice-theoretical fixpoint theorem and its applications, Pacific
	J. Math.~5 (1955) 285--309.
	
	\bibitem{wolf} S. Wolfram, Cellular Automata and Complexity, Addison-Wesley, New York,~1994.
\end{thebibliography}
\end{document}